\documentclass[12pt]{amsart}

\usepackage{hyperref}
\usepackage{amssymb,latexsym, amsthm, enumerate, mathptmx}

\usepackage{tikz}
\usetikzlibrary{arrows}
\usepackage[all]{xy}
\usepackage{color}

\newtheorem{theorem}{Theorem}[section]
\newtheorem{lemma}[theorem]{Lemma}
\newtheorem{proposition}[theorem]{Proposition}
\newtheorem{corollary}[theorem]{Corollary}
\newtheorem*{claim}{Claim}
\newtheorem{question}[theorem]{Question}

\newtheoremstyle{definition}
  {6pt}
  {6pt}
  {}
  {}
  {\bfseries}
  {.}
  {.5em}
  {}%

\theoremstyle{definition}
\newtheorem{definition}[theorem]{Definition}

\newtheoremstyle{remark}
  {6pt}
  {6pt}
  {}
  {}
  {\bfseries}
  {.}
  {.5em}
  {}%

\theoremstyle{remark}
\newtheorem{remark}[theorem]{Remark}

\newtheoremstyle{example}
  {6pt}
  {6pt}
  {}
  {}
  {\bfseries}
  {.}
  {.5em}
  {}%

\theoremstyle{example}
\newtheorem{example}[theorem]{Example}

\makeatletter
\renewcommand\@makefntext[1]{%
\setlength\parindent{1em}%
\noindent
\makebox[1.8em][r]{}{#1}}
\makeatother

\DeclareMathOperator{\codim}{codim}
\DeclareMathOperator{\supp}{supp}
\DeclareMathOperator{\length}{length}
\DeclareMathOperator{\rank}{rank}
\DeclareMathOperator{\ch}{char}

\title[On the Orlik--Terao ideal and the relation space of a hyperplane arrangement]{On the Orlik--Terao ideal and the relation space\\
of a hyperplane arrangement}

\author{Dinh Van Le}
\address{Universit\"at Osnabr\"uck, Institut f\"ur Mathematik, 49069 Osnabr\"uck, Germany}
\email{dlevan@uos.de}

\author{FATEMEH MOHAMMADI}
\address{Institute of Science and Technology Austria,
3400 Klosterneuburg, Austria}
\email{fatemeh.mohammadi@ist.ac.at}

\date{2 December 2014}

\tikzstyle{Cwhite}=[scale = .8,circle, fill = white, minimum size=3mm]
\tikzstyle{Cgray}=[scale = .4,circle, fill = gray, minimum size=3mm]
\tikzstyle{Cblack2}=[scale = .2,circle, fill = black, minimum size=5mm]
\tikzstyle{Cblack}=[scale = .7,circle, fill = black, minimum size=3mm]
\tikzstyle{C0}=[scale = .9,circle, fill = black!0, inner sep = 0pt, minimum size=3mm]
\tikzstyle{C1}=[scale = .7,circle, fill = black!0, inner sep = 0pt, minimum size=3mm]
\tikzstyle{Cred}=[scale = .4,circle, fill = red, minimum size=3mm]
\tikzstyle{Cblue}=[scale = .4,circle, fill =blue, minimum size=3mm]

\begin{document}

\begin{abstract}
 The relation space of a hyperplane arrangement is the vector space of all linear dependencies among the defining forms of the hyperplanes in the arrangement. In this paper, we study the relationship between the relation space and the Orlik--Terao ideal of an arrangement. In particular, we characterize spanning sets of the relation space in terms of the Orlik--Terao ideal. This result generalizes a characterization of 2-formal arrangements due to Schenck and Toh\v{a}neanu \cite[Theorem 2.3]{ST}. We also study the minimal prime ideals of subideals of the Orlik--Terao ideal associated to subsets of the relation space. Finally, we give examples to show that for a 2-formal arrangement, the codimension of the Orlik--Terao ideal is not necessarily equal to that of its subideal generated by the quadratic elements.
\end{abstract}

\maketitle

\section{Introduction}

The relation space of a hyperplane arrangement consists of linear relations among the defining forms of the hyperplanes in the arrangement. An arrangement is called \emph{2-formal} if its relation space is spanned by the relations of length 3. 2-formality is interesting because it is a common property of several important classes of arrangements: Falk and Randell \cite{FR1} introduced the notion of 2-formality and proved that $K(\pi,1)$ arrangements, rational $K(\pi,1)$ arrangements and arrangements with quadratic Orlik--Solomon algebra are 2-formal. They also conjectured that 2-formality holds for free arrangements. This conjecture was later resolved by Yuzvinsky \cite{Y3}, and a different proof for the result was given by Brandt and Terao \cite{BT}. In \cite{BB}, Bayer and Brandt showed that discriminantal arrangements over generic arrangements are 2-formal. Various combinatorial conditions which yield 2-formality were introduced by Falk \cite{F}. It should be mentioned that 2-formality is not 
combinatorially determined:
Yuzvinsky \cite{Y3} gave examples of two arrangements with the same intersection lattice, one of which is 2-formal while the other is not.

In \cite{ST}, Schenck and Toh\v{a}neanu explored the relationship between the Orlik--Terao ideal and the relation space by characterizing 2-formality in terms of the Orlik--Terao ideal. The goal of this paper is to study further connections between these two objects. First of all, we give in Section 3 characterizations of spanning sets of the relation space in terms of the Orlik--Terao ideal (Theorem \ref{th31}, Corollary \ref{co34}). As direct consequences, we obtain characterizations of 2-formal arrangements (Corollary \ref{co37}) as well as $k$-generated arrangements defined in \cite{BT} (Corollary \ref{co36}). The former one recovers one result of Schenck and Toh\v{a}neanu \cite[Theorem 2.3]{ST}. In Remark \ref{rm38} we clarify the proof of \cite[Theorem 2.3]{ST} by filling in some details.

Section 4 is devoted to the study of the minimal prime ideals of subideals of the Orlik--Terao ideal associated to subsets of the relation space (Theorem \ref{th41}, Proposition \ref{pr44}). We also give a sufficient condition for such a subideal to be prime (Proposition \ref{pr47}).

In the last section we provide examples to demonstrate that the codimension of the quadratic Orlik--Terao ideal is not governed by 2-formality (Examples \ref{ex51}, \ref{ex52}, \ref{ex53}). In fact, we show that in the class of 2-formal arrangements there does not exist a linear bound for the codimension of the Orlik--Terao ideal in terms of the codimension of the quadratic Orlik--Terao ideal (Corollary \ref{co54}).
This leads us to a question regarding the relationship between a free or $K(\pi,1)$ arrangement and the codimension of its quadratic Orlik--Terao ideal (Question \ref{qu55}).

Throughout the paper, unless otherwise stated, the following notation will be used. Let $\mathcal{A}=\{H_1,\ldots,H_n\}$ be a central hyperplane arrangement of rank $r_{\mathcal{A}}$ in a vector space $V$ over some field ${K}$. For $i=1,\ldots,n$, let $\alpha_i$ be the form in the dual space $V^*$ of $V$ that defines $H_i$, i.e., $\ker \alpha_i=H_i$. Let $S=K[x_1,\ldots,x_n]$ denote the polynomial ring in $n$ variables over $K$. For unexplained terminology, we refer to the book by Orlik and Terao \cite{OT}.

\section{Background}

In the present section we briefly review the relation space and the Orlik--Terao ideal of a hyperplane arrangement. A generalization of a construction due to Yuzvinsky \cite{Y3}, which plays an important role in this paper, will also be discussed.

\subsection{The relation space}
We adopt the terminology from \cite{BT}. The {\it relation space} of $\mathcal{A}$, denoted $F(\mathcal{A})$, is the kernel of the following ${K}$-linear map
$$S_1=\bigoplus_{i=1}^n{K}x_i\rightarrow V^*,\ x_i\mapsto \alpha_i\ \text{ for }\ i=1,\ldots,n.$$
Note that the above map has rank $r_{\mathcal{A}}$. Thus $\dim_K F(\mathcal{A})=n-r_{\mathcal{A}}$.

Elements of $F(\mathcal{A})$ are called {\it relations}.  Apparently, relations come from dependencies among hyperplanes in $\mathcal{A}$: whenever $\{H_{i_1},\ldots,H_{i_m}\}$ is a dependent subset of $\mathcal{A}$ and $a_t\in{K}$ are constants such that $\sum_{t=1}^ma_t\alpha_{i_t}=0$, then $r=\sum_{t=1}^ma_tx_{i_t}$ is a relation. In particular, every circuit $C$ of $\mathcal{A}$ gives rise to a unique (up to a constant) relation $r_C$. A trivial but useful fact is that $x_i\not\in F(\mathcal{A})$ for $i=1,\ldots,n.$ Given a relation $r=\sum_{i=1}^na_ix_{i}$, the \emph{support} of $r$ is the set $\supp(r)=\{i\mid a_i\ne 0\}$. The \emph{length} of $r$, denoted $\length(r)$, is the cardinality of its support.

In order to state the results in the next section, we need a slight generalization of the notions of formal and $k$-generated arrangements introduced in \cite{FR1} and \cite{BT} respectively.

\begin{definition}
 Let $\mathfrak{R}$ be a subset of $F(\mathcal{A})$. We say that $\mathcal{A}$ is $\mathfrak{R}$-\emph{generated} if $F(\mathcal{A})$ is spanned by $\mathfrak{R}$. If $\mathfrak{R}$ is the set of relations of length at most $k+1$ for some $k\geq 2$, then $\mathcal{A}$ is called $k$-\emph{generated}. 2-generated arrangements are also called \emph{2-formal} arrangements.
\end{definition}

\subsection{The Orlik--Terao algebra}

For a polynomial $f\in S$, let $\Lambda(f)$ be the least common multiple of the monomials of $f$. (By convention, $\Lambda(0)=1$.) Then the evaluation $f(x_1^{-1},\ldots,x_n^{-1})$ of $f$ at $(x_1^{-1},\ldots,x_n^{-1})$ can be uniquely written in the form
$$f(x_1^{-1},\ldots,x_n^{-1})=\frac{\iota(f)(x_1,\ldots,x_n)}{\Lambda(f)(x_1,\ldots,x_n)}\quad \text{with}\quad \iota(f)(x_1,\ldots,x_n)\in S.$$
Thus we obtain a map $\iota:S\rightarrow S,\ f\mapsto \iota(f)$. For example, if $r=\sum_{i=1}^na_ix_{i}\in F(\mathcal{A})$, then $\Lambda(r)=x_{\supp(r)}$ and
$$\iota(r)=\sum_{i\in\supp(r)}a_ix_{\supp(r)-i}.$$
Here, for a subset $\Gamma$ of $[n]:=\{1,\ldots,n\}$ we write $x_\Gamma=\prod_{i\in\Gamma}x_i$. Note that $\iota(r)$ is a homogeneous polynomial of degree $\length(r)-1.$

The straightforward proof of the following lemma is omitted.

 \begin{lemma}\label{lm22}
  The map $\iota:S\rightarrow S$ described above has the following properties
 $$\begin{aligned}
 \iota(f)(x_1^{-1},\ldots,x_n^{-1})&=\frac{f(x_1,\ldots,x_n)}{\Lambda(f)(x_1,\ldots,x_n)},\\
 \iota(fg)&=\iota(f)\iota(g)
 \end{aligned}$$
 for all $f,g\in S$.
 \end{lemma}

\begin{definition}
Let $\mathcal{A}$ be a central arrangement. We call $I(\mathcal{A})=(\iota(r)\mid r\in F(\mathcal{A}))$ the {\it Orlik--Terao ideal}, and $\mathbf{C}(\mathcal{A})=S/I(\mathcal{A})$ the {\it Orlik--Terao algebra} of $\mathcal{A}$.
\end{definition}

The Orlik--Terao algebra was introduced by Orlik and Terao in \cite{OT2}. This algebra, which can be viewed as a commutative analogue of the well-studied Orlik--Solomon algebra \cite{OS}, has been the topic of a number of recent researches \cite{BP,DGT,Le,LR,MP,PS,SSV,Sc,ST,T}. We quote here two basic facts about the Orlik--Terao ideal that are relevant to our work. Recall that every circuit $C$ of $\mathcal{A}$ may be assigned to a unique relation $r_C\in F(\mathcal{A})$. Let $\mathcal{C}(F(\mathcal{A}))$ denote the set of relations coming from circuits by this way. Note that $\mathcal{C}(F(\mathcal{A}))$ is a finite subset of $F(\mathcal{A})$.

\begin{theorem}[{\cite[Theorem 4]{PS}}]\label{th24}
 Let $\mathcal{A}$ be a central arrangement. Then the set $\{\iota(r)\mid r\in\mathcal{C}(F(\mathcal{A}))\}$ is a universal Gr\"{o}bner basis for the Orlik--Terao ideal $I(\mathcal{A})$. In particular, $I(\mathcal{A})=(\iota(r)\mid r\in \mathcal{C}(F(\mathcal{A}))$.
\end{theorem}

\begin{theorem}[{\cite[Proposition 2.1]{ST}}]\label{th25}
Let $\mathcal{A}$ be a rank $r_{\mathcal{A}}$ central arrangement of $n$ hyperplanes. Then $I(\mathcal{A})$ is a prime ideal in $S$ of codimension $n-r_{\mathcal{A}}\ (=\dim_K F(\mathcal{A}))$. Moreover, $I(\mathcal{A})$ contains no linear forms of $S$.
\end{theorem}

\subsection{Arrangements from the relation space}

To each subset $\mathfrak{R}$ of the relation space $F(\mathcal{A})$, we will associate a hyperplane arrangement which is $\mathfrak{R}$-generated. This construction, which was first considered by Yuzvinsky \cite{Y3}, turns out to be very useful for studying spanning sets of $F(\mathcal{A})$ as well as subideals of the Orlik--Terao ideal. We follow the presentation of Falk \cite{F}.

By abuse of notation, we will identify $S_1=\bigoplus_{i=1}^n{K}x_i$ with the dual space $(K^n)^*$ of $K^n$. Then every relation $r\in F(\mathcal{A})$ defines a function on $K^n$. Let $\mathfrak{R}$ be a subset of $F(\mathcal{A})$, and let $Z(\mathfrak{R})$ be the zero locus of $\mathfrak{R}$, i.e.,
$$Z(\mathfrak{R})=\bigcap_{r\in\mathfrak{R}}\ker(r)=\{v\in K^n\mid r(v)=0\ \text{for all}\ r \in\mathfrak{R}\}.$$
Note that $Z(\mathfrak{R})$ is a linear subspace of $K^n$. Now intersecting $Z(\mathfrak{R})$ with the coordinate hyperplanes in $K^n$ we obtain the following hyperplane arrangement in $Z(\mathfrak{R})$:
$$\mathcal{A}(\mathfrak{R})=\{\tilde{H}_1,\ldots,\tilde{H}_n\},\quad \text{where}\quad \tilde{H}_i=\ker(x_i)\cap Z(\mathfrak{R}).$$
It should be noted that $\tilde{H}_i$ are indeed hyperplanes in $Z(\mathfrak{R})$: one has $Z(\mathfrak{R})\nsubseteq \ker(x_i)$ because $F(\mathcal{A})$ contains no variables.

\begin{proposition}\label{pr26}
 With notation as above, the following statements hold.
\begin{enumerate}
 \item
The relation space of $\mathcal{A}(\mathfrak{R})$ is $F(\mathcal{A}(\mathfrak{R}))=K\mathfrak{R}$, where $K\mathfrak{R}$ is the $K$-subspace of $F(\mathcal{A})$ spanned by $\mathfrak{R}$. In particular, the arrangement $\mathcal{A}(\mathfrak{R})$ is $\mathfrak{R}$-generated.

\item
$I(\mathcal{A}(\mathfrak{R}))=(\iota(r)\mid r\in K\mathfrak{R})=(\iota(r)\mid r\in \mathcal{C}(K\mathfrak{R}))$ and so
$$\codim I(\mathcal{A}(\mathfrak{R}))=\dim_K K\mathfrak{R}.$$
\end{enumerate}
\end{proposition}

\begin{proof}
 Since (ii) follows immediately from (i), Theorem \ref{th24} and Theorem \ref{th25}, it suffices to prove (i). By definition, $\tilde{H}_i=\ker(x_i\restriction_{Z(\mathfrak{R})})$, where $x_i\restriction_{Z(\mathfrak{R})}$ is the restriction of the function $x_i\in (K^n)^*$ to $Z(\mathfrak{R})$. Thus $F(\mathcal{A}(\mathfrak{R}))$ is the kernel of the map
$$S_1\rightarrow Z(\mathfrak{R})^*,\ x_i\mapsto x_i\restriction_{Z(\mathfrak{R})}\ \text{ for }\ i=1,\ldots,n.$$
Hence
$$\begin{aligned}
   F(\mathcal{A}(\mathfrak{R}))&=\{r\in S_1\mid\ \  r\restriction_{Z(\mathfrak{R})}=0\}=\{r\in S_1\mid\ r(v)=0\ \text{for all}\ v\in Z(\mathfrak{R})\}=K\mathfrak{R},
  \end{aligned}
$$
as claimed.
\end{proof}

\section{The relation space and the Orlik--Terao ideal}

In this section the relationship between the relation space and the Orlik--Terao ideal will be analyzed in more details. The main result is a characterization of spanning sets of the relation space in terms of the Orlik--Terao ideal. This generalizes the characterization of 2-formal arrangements given in \cite[Theorem 2.3]{ST}.

Let $\mathcal{A}$ be a central arrangement as before. Let $\mathfrak{R}$ be a subset of the relation space $F(\mathcal{A})$. The subspace spanned by $\mathfrak{R}$ is, as in the previous section, denoted by $K\mathfrak{R}$. We will be concerned with the following ideals in $S$:
\[
J(\mathfrak{R})=(\iota(r)\mid r\in\mathfrak{R})\quad \text{and}\quad I(\mathfrak{R})=(\iota(r)\mid r\in K\mathfrak{R}).
\]
 Obviously,  $J(\mathfrak{R})\subseteq I(\mathfrak{R})$ and both of them are subideals of the Orlik--Terao ideal of $\mathcal{A}$. Also, $I(\mathfrak{R})$ is the Orlik--Terao ideal of the arrangement $\mathcal{A}(\mathfrak{R})$ considered in Proposition \ref{pr26}.

Recall that $x_{[n]}=x_1\cdots x_n$. The main result of this section is as follows.

\begin{theorem}\label{th31}
Let $\mathcal{A}$ be a central hyperplane arrangement and let $\mathfrak{R}\subseteq F(\mathcal{A})$. Then the following conditions are equivalent:
\begin{enumerate}
\item
$\mathcal{A}$ is $\mathfrak{R}$-{generated};

\item
$I(\mathcal{A})={J(\mathfrak{R})}:x_{[n]};$

\item
$I(\mathcal{A})=\sqrt{J(\mathfrak{R})}:x_{[n]}.$
\end{enumerate}
\end{theorem}

Let $P(\mathfrak{R})=(r\mid r\in \mathfrak{R})$ be the ideal in $S$ generated by the elements of $\mathfrak{R}$. Then the degree 1 component of $P(\mathfrak{R})$ is $K\mathfrak{R}.$ In particular, $P(\mathfrak{R})$ contains no variables as so does the relation space $F(\mathcal{A})$. Moreover, since $P(\mathfrak{R})$ is generated by linear forms, it is a prime ideal. The proof of the previous theorem is relied on the following lemma.

\begin{lemma}\label{lm32}
The following conditions are equivalent for a polynomial $f$ in $S$:
\begin{enumerate}
\item
$f\in P(\mathfrak{R})$;

\item
$\iota(f)\in \sqrt{J(\mathfrak{R})}:x_{[n]}.$
\end{enumerate}
If, in addition, $f$ is a linear form, then each of the above conditions is equivalent to any one of the following:
\begin{enumerate}[\rm(iii)]
\item
$f\in K\mathfrak{R}$;

\item[\rm(iv)]
$\iota(f)\in {J(\mathfrak{R})}:x_{[n]}.$
\end{enumerate}
\end{lemma}

\begin{proof}
(i)$\Rightarrow$(ii): Assume $f\in P(\mathfrak{R})$. Then we may write
\begin{equation}\label{eq55}
f=\sum_{r\in \mathfrak{R}} f_r r \quad \text{with} \quad f_r\in S.
 \end{equation}
 Evaluating this expression at $(x_1^{-1},\ldots,x_n^{-1})$ we get
\begin{equation}\label{eq552}
\frac{\iota(f)}{\Lambda(f)}=\sum_{r\in \mathfrak{R}}\frac{\iota(f_r)}{\Lambda({f_r})}\frac{\iota(r)}{\Lambda({r})}.
\end{equation}
Since $\Lambda({f_r}),\Lambda({r})$ are monomials, the above equality implies that $x_{[n]}^l\iota(f)\in J(\mathfrak{R})$ for some $l\geq 1$. Thus $(x_{[n]}\iota(f))^l\in J(\mathfrak{R})$, and hence $\iota(f)\in \sqrt{ J(\mathfrak{R})}:x_{[n]}$.

(ii)$\Rightarrow$(i): If $\iota(f)\in \sqrt{J(\mathfrak{R})}:x_{[n]}$, then there exists $l\geq 1$ such that
$$(x_{[n]} \iota(f))^l=\sum_{r\in \mathfrak{R}} g_r \iota(r),\quad \text{where}\quad g_r\in S.$$
 Evaluating the latter expression at $(x_1^{-1},\ldots,x_n^{-1})$ and noting Lemma \ref{lm22} we obtain
$$\Big(\frac{f}{x_{[n]}\Lambda(f)}\Big)^l=\sum_{r\in\mathfrak{R}}\frac{\iota(g_r)}{\Lambda({g_r})}\frac{r}{\Lambda({r})}.$$
It follows that $x_{[n]}^mf^l\in P(\mathfrak{R})$ for some $m\geq 1$. Since $P(\mathfrak{R})$ is prime and $x_i\not\in P(\mathfrak{R})$ for $i=1,\ldots,n$, we conclude that $f\in P(\mathfrak{R})$.

Now suppose that $f$ is a linear form. Then we immediately have (i)$\Leftrightarrow$(iii). The implication (iv)$\Rightarrow$(ii) is also clear. To complete the proof, we will show (i)$\Rightarrow$(iv). The argument is similar to the proof of (i)$\Rightarrow$(ii) with some minor changes. Since $f$ is a linear form, the polynomials $f_r$ in the representation \eqref{eq55} of $f$ can be chosen in $K$. Then $\Lambda({f_r})=1$ for every $r\in \mathfrak{R}$. Each relation $r$ is a linear form, so $\Lambda({r})$ divides $x_{[n]}$. It now follows from \eqref{eq552} that $x_{[n]}\iota(f)\in J(\mathfrak{R})$, whence $\iota(f)\in J(\mathfrak{R}):x_{[n]}.$
\end{proof}

We are now ready to prove Theorem \ref{th31}.

\begin{proof}[Proof of Theorem \ref{th31}]
First, recall that $J(\mathfrak{R})\subseteq I(\mathcal{A})$ and $I(\mathcal{A})$ is a prime ideal containing no variables (see Theorem \ref{th25}), hence
\begin{equation}\label{eq54}
J(\mathfrak{R}):x_{[n]}\subseteq \sqrt{J(\mathfrak{R})}:x_{[n]}\subseteq I(\mathcal{A}):x_{[n]}=I(\mathcal{A}).
\end{equation}

(i)$\Rightarrow$(ii): Let $r\in F(\mathcal{A})$. As $\mathcal{A}$ is $\mathfrak{R}$-generated, $r\in K\mathfrak{R}.$ So by Lemma \ref{lm32}, $\iota(r)\in J(\mathfrak{R}):x_{[n]}$. It follows that $I(\mathcal{A})\subseteq J(\mathfrak{R}):x_{[n]}$. Now from \eqref{eq54} we get $I(\mathcal{A})=J(\mathfrak{R}):x_{[n]}.$

(ii)$\Rightarrow$(iii): This is immediate from \eqref{eq54}.

(iii)$\Rightarrow$(i): Let $r\in F(\mathcal{A})$. Since $\iota(r)\in I(\mathcal{A})=\sqrt{J(\mathfrak{R})}:x_{[n]}$, it follows from Lemma \ref{lm32} that $r\in K\mathfrak{R}.$ Thus $\mathcal{A}$ is $\mathfrak{R}$-generated.
\end{proof}

Combining Proposition \ref{pr26} and Theorem \ref{th31} we obtain the following important and interesting corollary.

\begin{corollary}\label{co33}
 Let $\mathcal{A}$ be a central hyperplane arrangement and let $\mathfrak{R}\subseteq F(\mathcal{A})$. Then $$I(\mathfrak{R})=\sqrt{J(\mathfrak{R})}:x_{[n]}={J(\mathfrak{R})}:x_{[n]}.$$
 Moreover, $I(\mathfrak{R})$ is the unique associated prime ideal of $J(\mathfrak{R})$ which does not contain any variables.
\end{corollary}

\begin{proof}
 Recall from Proposition \ref{pr26} that $I(\mathfrak{R})$ is the Orlik--Terao ideal of the arrangement $\mathcal{A}(\mathfrak{R})$ which is $\mathfrak{R}$-generated. Thus the first assertion of the corollary is a consequence of Theorem \ref{th31}. The second assertion is deduced from Theorem \ref{th25} and the fact that
\begin{equation}\label{eq34}
 J(\mathfrak{R})=I(\mathfrak{R})\cap(J(\mathfrak{R})+x_{[n]})
\end{equation}
which follows easily from the first assertion.
\end{proof}

In the sequel, we use the standard notation from algebraic geometry.

\begin{corollary}\label{co34}
 Let $\mathcal{A}$ be a central hyperplane arrangement and let $\mathfrak{R}\subseteq F(\mathcal{A})$. Then the following conditions are equivalent:
\begin{enumerate}
 \item
$\mathcal{A}$ is $\mathfrak{R}$-generated;

\item
$\codim (I(\mathcal{A}))=\codim(J(\mathfrak{R}):x_{[n]}).$
\end{enumerate}
 If, in addition, the field $K$ is algebraically closed, then each of the above conditions is equivalent to the following one:
\begin{enumerate}
 \item[\rm(iii)]
$\codim \mathbf{V}(I(\mathcal{A}))=\codim(\mathbf{V}(J(\mathfrak{R}))\cap(K^*)^n).$
\end{enumerate}

\end{corollary}

\begin{proof}
 By Corollary \ref{co33}, we know that $J(\mathfrak{R}):x_{[n]}$ is a prime ideal. So the equivalence (i)$\Leftrightarrow$(ii) follows immediately from Theorem \ref{th31}. When $K$ is an algebraically closed field, we obtain (ii)$\Leftrightarrow$(iii) from the following computation:
$$\begin{aligned}
   \mathbf{I}(\mathbf{V}(J(\mathfrak{R}))\cap(K^*)^{n})&=\mathbf{I}(\mathbf{V}(J(\mathfrak{R}))-\mathbf{V}(x_{[n]}))
=\mathbf{I}(\mathbf{V}(J(\mathfrak{R})):\mathbf{I}(\mathbf{V}(x_{[n]}))\\
&=\sqrt{J(\mathfrak{R})}:x_{[n]}=J(\mathfrak{R}):x_{[n]},
  \end{aligned}$$
see, e.g., \cite[Theorem 4.2.6, Corollary 4.4.8]{CLO}; the last equality holds by Corollary \ref{co33}.
\end{proof}

\begin{remark}\label{rm35}
 The proof of Corollary \ref{co34} shows, when $K$ is an algebraically closed field, the irreducible variety $\mathbf{V}(I(\mathfrak{R}))$ is the Zariski closure of $\mathbf{V}(J(\mathfrak{R}))\cap(K^*)^{n}$. Hence, in particular, $\codim(\mathbf{V}(J(\mathfrak{R}))\cap(K^*)^n)=\codim (I(\mathfrak{R}))=\dim_K K\mathfrak{R}$ (see Proposition \ref{pr26}).
\end{remark}

As a direct consequence of Theorem \ref{th31} and Corollary \ref{co34}, the following result characterizes $k$-generated arrangements in terms of the Orlik--Terao ideal.

\begin{corollary}\label{co36}
 Let $\mathcal{A}$ be a central arrangement. Given a number $k\geq 2$, let $I_{\leq k}(\mathcal{A})$ be the ideal in $S$ generated by all elements of the Orlik--Terao ideal $I(\mathcal{A})$ of degree at most $k$. Then the following conditions are equivalent:
\begin{enumerate}
\item
$\mathcal{A}$ is $k$-{generated};

\item
$I(\mathcal{A})={I_{\leq k}(\mathcal{A})}:x_{[n]};$

\item
$\codim I(\mathcal{A})=\codim(I_{\leq k}(\mathcal{A}):x_{[n]}).$
\end{enumerate}
If the field $K$ is algebraically closed, then each of the above conditions is equivalent to the following one:
\begin{enumerate}
 \item [\rm(iv)]
$\codim \mathbf{V}(I(\mathcal{A}))=\codim(\mathbf{V}(I_{\leq k}(\mathcal{A}))\cap(K^*)^n).$
\end{enumerate}
\end{corollary}

Setting $k=2$ in the preceding corollary, we obtain a characterization of 2-formal arrangements which recovers \cite[Theorem 2.3]{ST}.

\begin{corollary}\label{co37}
 Let $I_{\langle2\rangle}(\mathcal{A})$ be the ideal generated by the quadratic elements of $I(\mathcal{A})$. Then the following conditions are equivalent:
\begin{enumerate}
\item
$\mathcal{A}$ is 2-formal;

\item
$I(\mathcal{A})={I_{\langle 2\rangle}(\mathcal{A})}:x_{[n]};$

\item
$\codim I(\mathcal{A})=\codim(I_{\langle2\rangle}(\mathcal{A}):x_{[n]}).$
\end{enumerate}
If the field $K$ is algebraically closed, then each of the above conditions is equivalent to the following one:
\begin{enumerate}
 \item [\rm(iv)]
$\codim \mathbf{V}(I(\mathcal{A}))=\codim(\mathbf{V}(I_{\langle2\rangle}(\mathcal{A}))\cap(K^*)^n).$
\end{enumerate}
\end{corollary}

\begin{remark}\label{rm38}
(i) Suppose that the field $K$ is algebraically closed. Then \cite[Theorem 2.3]{ST} asserts the equivalence of conditions (i) and (iv) of Corollary \ref{co37}, and the idea of the proof is to show that
\begin{equation}\label{eq32}
 \codim(\mathbf{V}(I_{\langle2\rangle})\cap(K^*)^n)=\dim_K F_2,
\end{equation}
where $I_{\langle2\rangle}:=I_{\langle 2\rangle}(\mathcal{A})$ and $F_2$ is the $K$-subspace of $F(\mathcal{A})$ spanned by the length 3 relations. In order to prove \eqref{eq32}, the authors of \cite{ST} have given a nice argument for the crucial fact that
$$\rank \mathbf{J}_p(I_{\langle2\rangle})=\dim_K F_2$$
 for every point $p\in \mathbf{V}(I_{\langle 2\rangle})\cap(K^*)^{n}$, where $\mathbf{J}_p(I_{\langle2\rangle})$ is the Jacobian matrix of $I_{\langle2\rangle}$ at $p$. However, to complete the proof of \cite[Theorem 2.3]{ST} one needs to show further that
\begin{equation}\label{eq31}
 \codim(\mathbf{V}(I_{\langle 2\rangle})\cap(K^*)^{n})=\rank \mathbf{J}_p(I_{\langle2\rangle}).
\end{equation}
Recall the following well-known facts (see, e.g., \cite[Section 9.6]{CLO}):
 $$\begin{aligned}
    \dim(\mathbf{V}(I_{\langle 2\rangle})\cap(K^*)^{n})&=\max\{\dim_p\mathbf{V}(I_{\langle 2\rangle})\mid p\in\mathbf{V}(I_{\langle 2\rangle})\cap(K^*)^{n}\}\\
&=\max\{\dim_K T_p(\mathbf{V}(I_{\langle 2\rangle}))\mid p\in\mathbf{V}(I_{\langle 2\rangle})\cap(K^*)^{n}, p\ \text{is smooth}\}\\
&=\max\{n-\rank \mathbf{J}_p(\sqrt{I_{\langle 2\rangle}})\mid p\in\mathbf{V}(I_{\langle 2\rangle})\cap(K^*)^{n}, p\ \text{is smooth}\},
   \end{aligned}$$
 where $T_p(\mathbf{V}(I_{\langle 2\rangle}))$ denotes the tangent space of $\mathbf{V}(I_{\langle 2\rangle})$ at $p$. Thus \eqref{eq31} will follow if one can show that
\begin{equation}\label{eq33}
 \rank \mathbf{J}_p(\sqrt{I_{\langle 2\rangle}})=\rank \mathbf{J}_p(I_{\langle2\rangle})\quad \text{for every}\quad p\in \mathbf{V}(I_{\langle 2\rangle})\cap(K^*)^{n}.
\end{equation}
 This fact is, however, not so obvious. It should be mentioned here that $I_{\langle 2\rangle}$ is not a radical ideal in general (see Example \ref{ex51}), and that \eqref{eq33} is not true if $I_{\langle 2\rangle}$ is replaced by an arbitrary ideal of $S$, e.g., $I=(x_1^2)$.

(ii) One may prove \eqref{eq33}, and thereby complete the original proof of \cite[Theorem 2.3]{ST}, by using Corollary \ref{co33} and the following simple observation.

\begin{claim}
 Let $I,I'$ be ideals in $S$. Then $\rank \mathbf{J}_p(I:I')=\rank \mathbf{J}_p(I)$ for every $p\in \mathbf{V}(I)-\mathbf{V}(I').$
\end{claim}

\begin{proof}
 Let $p\in \mathbf{V}(I)-\mathbf{V}(I').$ Then there is $g\in I'$ such that $g(p)\ne 0$. For any $f\in I:I'$, $gf\in I$. So we may write
$gf=\sum_{j=1}^mg_jf_j$ with $g_j\in S, f_j\in I.$ Since $f_j(p)=0$ for $j=1,\ldots,m$, we deduce that $f(p)=0$ and
$$g(p)\frac{\partial f}{\partial x_i}(p)=\sum_{j=1}^m g_j(p)\frac{\partial f_j}{\partial x_i}(p),\ i=1,\ldots,n.$$
Thus, $\rank \mathbf{J}_p(I:I')=\rank \mathbf{J}_p(I)$.
\end{proof}

Now by Corollary \ref{co33} and the above claim we obtain
$$\rank \mathbf{J}_p(\sqrt{I_{\langle 2\rangle}})=\rank \mathbf{J}_p(\sqrt{I_{\langle 2\rangle}}:x_{[n]})=\rank \mathbf{J}_p({I_{\langle 2\rangle}}:x_{[n]})=\rank \mathbf{J}_p(I_{\langle2\rangle})$$
for every $p\in \mathbf{V}(I_{\langle 2\rangle})\cap(K^*)^{n}.$
\end{remark}

\section{Minimal prime ideals of subideals of the Orlik--Terao ideal}

We keep the notation of the previous sections. Let $\mathcal{A}$ be a central arrangement. In this section we describe, for a given subset $\mathfrak{R}$ of the relation space $F(\mathcal{A})$, the minimal prime ideals of the ideal $J(\mathfrak{R})$. As an application, we give an instance in which $J(\mathfrak{R})$ is a prime ideal.

Recall from Corollary \ref{co33} that $I(\mathfrak{R})$ is the only associated prime ideal of $J(\mathfrak{R})$ which does not contain any variables. To determine other minimal prime ideals of $J(\mathfrak{R})$ we need some notation. We call a subset $\Gamma$ of $[n]$ an $\mathfrak{R}$-\emph{cover} if either $|\Gamma\cap \supp(r)|=0$ or $\geq 2$ for all $r\in\mathfrak{R}.$ The set of all $\mathfrak{R}$-covers is denoted by $co(\mathfrak{R})$. Suppose $\Gamma\in co(\mathfrak{R})$. We set
$$\mathfrak{R}_0(\Gamma)=\{r\in \mathfrak{R}\mid\ \  |\Gamma\cap \supp(r)|=0\}\quad \text{and}\quad \mathfrak{R}_+(\Gamma)=\mathfrak{R}-\mathfrak{R}_0(\Gamma).$$
Consider the ideal $Q_\Gamma(\mathfrak{R})=(x_i\mid i\in \Gamma)+I(\mathfrak{R}_0(\Gamma))$ in $S$. Recall that, as shown in Proposition \ref{pr26}, $I(\mathfrak{R}_0(\Gamma))=(\iota(r)\mid r\in K\mathfrak{R}_0(\Gamma))$ is the Orlik--Terao ideal of the arrangement $\mathcal{A}({R}_0(\Gamma))$. Thus $Q_\Gamma(\mathfrak{R})$ is a prime ideal because it is the sum of two prime ideals in disjoint sets of variables. Now if $r\in \mathfrak{R}_+(\Gamma)$, then $|\Gamma\cap \supp(r)|\geq 2$, and so $\iota(r)\in (x_i\mid i\in \Gamma)$. This yields $J(\mathfrak{R})\subseteq Q_\Gamma(\mathfrak{R})$. Note that by Proposition \ref{pr26},
$$\codim Q_\Gamma(\mathfrak{R})=\codim (x_i\mid i\in \Gamma)+\codim I(\mathfrak{R}_0(\Gamma)) =|\Gamma|+\dim_K K\mathfrak{R}_0(\Gamma).$$

The main result of this section embeds the minimal prime ideals of $J(\mathfrak{R})$ in a finite set, as follows.

\begin{theorem}\label{th41}
 Let $\mathfrak{R}$ be a subset of $F(\mathcal{A})$. Then the set of minimal prime ideals of $J(\mathfrak{R})$ coincides with the set of minimal elements of $\{Q_\Gamma(\mathfrak{R})\mid\Gamma\in co(\mathfrak{R})\}$. Thus, in particular,
$$\begin{aligned}
\sqrt{J(\mathfrak{R})}&=\bigcap_{\Gamma\in co(\mathfrak{R})}Q_\Gamma(\mathfrak{R}),\  \text{and}\\
\codim J(\mathfrak{R})&=\min \{|\Gamma|+\dim_K K\mathfrak{R}_0(\Gamma)\mid\Gamma\in co(\mathfrak{R})\}.
\end{aligned}$$
\end{theorem}

\begin{proof}
 It suffices to prove that every minimal prime ideal of $J(\mathfrak{R})$ is of the form $Q_\Gamma(\mathfrak{R})$ for some $\Gamma\in co(\mathfrak{R})$. First note that $I(\mathfrak{R})=Q_\emptyset(\mathfrak{R})$. Now suppose $Q$ is a minimal prime ideal of $J(\mathfrak{R})$ other than $I(\mathfrak{R})$. Then $Q$ contains variables, according to Corollary \ref{co33}. Let $\Gamma=\{i\in [n]\mid x_i\in Q\}$. We show that $\Gamma\in co(\mathfrak{R})$. Indeed, assume on the contrary that there exists $r\in \mathfrak{R}$ with $|\Gamma\cap \supp(r)|=1$, say $\Gamma\cap \supp(r)=\{j\}$. Writing $r=a_jx_j+\sum_{i\in \supp(r)-j}a_ix_i$ with $a_j\ne0$, one gets
  \begin{equation}\label{eq41}
  \iota(r)=a_jx_{\supp(r)-j}+x_j(\sum_{i\in \supp(r)-j}a_ix_{\supp(r)-\{i,j\}}).
  \end{equation}
  It follows that $x_{\supp(r)-j}\in Q$ since $x_j\in Q$ and $\iota(r)\in J(\mathfrak{R})\subseteq Q$. Thus, $x_i\in Q$ and hence $i\in \Gamma$ for some $i\in \supp(r)-j$. This is a contradiction to the assumption above. So we must have $\Gamma\in co(\mathfrak{R})$.

   Let us now prove $Q=Q_\Gamma(\mathfrak{R}).$ We may write $Q=(x_i\mid i\in \Gamma)+Q'$, where $Q'$ is an ideal generated by polynomials in $K[x_j\mid j\in [n]-\Gamma]$. Evidently, $Q'$ is a prime ideal containing no variables. From $J(\mathfrak{R})\subseteq Q$ it follows that $J(\mathfrak{R}_0(\Gamma))=(\iota(r)\mid r\in\mathfrak{R}_0(\Gamma))\subseteq Q'$. If there was some prime ideal $Q_1'$ with $J(\mathfrak{R}_0(\Gamma))\subseteq Q_1'\subsetneq Q'$, then $Q_1=(x_i\mid i\in \Gamma)+Q_1'$ would be a prime such that $J(\mathfrak{R}(\Gamma))\subseteq Q_1\subsetneq Q$. This contradicts the minimality of $Q$. Thus $Q'$ must be a minimal prime ideal of $J(\mathfrak{R}_0(\Gamma))$. Now by Corollary \ref{co33}, $Q'=I(\mathfrak{R}_0(\Gamma))$. Therefore, $Q=Q_\Gamma(\mathfrak{R}).$
\end{proof}

Because of the above theorem, we are now looking for minimal ideals of the form $Q_\Gamma(\mathfrak{R})$ with $\Gamma\in co(\mathfrak{R})$. First note that if  $Q_{\Gamma'}(\mathfrak{R})\subset Q_{\Gamma}(\mathfrak{R})$ for some $\Gamma,\Gamma'\in co(\mathfrak{R})$, then, \emph{a priori}, $\Gamma'\subset\Gamma$. Therefore, to check the minimality of $Q_\Gamma(\mathfrak{R})$ one only has to compare it with the ideals $Q_{\Gamma'}(\mathfrak{R})$, where $\Gamma'\subset\Gamma$. For doing this, we need to verify whether $I(\mathfrak{R}_0(\Gamma'))\not\subset Q_\Gamma(\mathfrak{R})$, or equivalently, whether $\iota(r)\not\in Q_\Gamma(\mathfrak{R})$ for some $r\in \mathcal{C}(K\mathfrak{R}_0(\Gamma'))$ (see Proposition \ref{pr26}). We can reduce slightly the number of verifications by a simple observation. We say that a relation $r$ is \emph{induced} from $\mathfrak{R}$ if there exist $r_i\in\mathfrak{R}$ and $a_i\in K$ for $i=1,\ldots,m$ such that
\begin{equation}\label{eq42}
 r=\sum_{i=1}^m a_ir_i,\quad\text{and}\quad |\supp(r_i)\cap(\bigcup_{j=1}^{i-1}\supp(r_j))|\leq1\ \text{for}\ i=2,\ldots,m.
\end{equation}

\begin{lemma}\label{lm42}
 If $r$ is induced from $\mathfrak{R}$, then $\iota(r)\in J(\mathfrak{R})$.
\end{lemma}

\begin{proof}
 By induction we may assume that $r$ has the representation \eqref{eq42} in which $m=2$. Let $\alpha=\supp(r_1)\cap\supp(r_2)$. By the assumption, either $\alpha=\emptyset$ or $\alpha=\{j\}$ for some $j\in [n]$. We then have
$$\supp(r_1)\cup\supp(r_2)-j\subseteq \supp(r)\subseteq \supp(r_1)\cup\supp(r_2).$$
If $\supp(r)= \supp(r_1)\cup\supp(r_2)$, then it follows immediately from the definition of the map $\iota$ that $$\iota(r)=a_1x_{\supp(r_2)-\alpha}\iota(r_1)+a_2x_{\supp(r_1)-\alpha}\iota(r_2).$$
Otherwise, when $\supp(r)= \supp(r_1)\cup\supp(r_2)-j$, we may write
$$a_1r_1=ax_j+\sum_{i\in\supp(r_1)-j}b_ix_i,\quad a_2r_2=-ax_j+\sum_{k\in\supp(r_2)-j}c_kx_k,$$
where $a,b_i,c_k\in K^*$. In this case,
$$\iota(r)=a^{-1}\Big(a_1\iota(r_1)\sum_{k\in\supp(r_2)-j}c_kx_{\supp(r_2)-\{j,k\}}+a_2\iota(r_2)\sum_{i\in\supp(r_1)-j}b_ix_{\supp(r_1)-\{j,i\}}\Big).$$
Thus we always get $\iota(r)\in J(\mathfrak{R})$, as claimed.
\end{proof}

Now denote by $n\mathcal{C}(K\mathfrak{R}_0(\Gamma'))$ the subset of $\mathcal{C}(K\mathfrak{R}_0(\Gamma'))$ consisting of the relations which are not induced from $\mathfrak{R}_0(\Gamma')$. Then from Lemma \ref{lm42} (and Proposition \ref{pr26}) we have
$$I(\mathfrak{R}_0(\Gamma'))=J(\mathfrak{R}_0(\Gamma'))+\big(\iota(r)\mid r\in n\mathcal{C}(K\mathfrak{R}_0(\Gamma'))\big).$$
Note that $J(\mathfrak{R}_0(\Gamma'))\subseteq J(\mathfrak{R})\subseteq Q_\Gamma(\mathfrak{R})$, so $I(\mathfrak{R}_0(\Gamma'))\not\subset Q_\Gamma(\mathfrak{R})$ if and only if there exists $r\in n\mathcal{C}(K\mathfrak{R}_0(\Gamma'))$ such that $\iota(r)\not\in Q_\Gamma(\mathfrak{R})$. For checking the last condition, we will use the following.

\begin{lemma}\label{lm43}
 Let $r\in F(\mathcal{A})$ and $\Gamma\in co(\mathfrak{R})$. Then $\iota(r)\not\in Q_\Gamma(\mathfrak{R})$ if and only if $r\not\in K\mathfrak{R}_0(\Gamma)$ and $|\Gamma\cap \supp(r)|\leq 1$.
\end{lemma}

\begin{proof}
 The ``only if'' part is clear. For the ``if'' part, assuming $\iota(r)\in Q_\Gamma(\mathfrak{R})$, we will show that if $|\Gamma\cap \supp(r)|\leq 1$, then $r\in K\mathfrak{R}_0(\Gamma)$. We first argue that the case $|\Gamma\cap \supp(r)|= 1$ is impossible.  Indeed, if $\Gamma\cap \supp(r)=\{j\}$, then we have an expression of $\iota(r)$ as in \eqref{eq41}. From $x_j,\iota(r)\in Q_\Gamma(\mathfrak{R})$ it follows that $x_{\supp(r)-j}\in Q_\Gamma(\mathfrak{R})$. But this cannot happen because $\Gamma\cap (\supp(r)-j)=\emptyset$ and $I(\mathfrak{R}_0(\Gamma))$ is a prime ideal containing no variables. Thus we must have $|\Gamma\cap \supp(r)|=0$. In this case, $\iota(r)\in I(\mathfrak{R}_0(\Gamma))=J(\mathfrak{R}_0(\Gamma)):x_{[n]}$, and so $r\in K\mathfrak{R}_0(\Gamma)$ by Lemma \ref{lm32}.
\end{proof}

From Lemma \ref{lm43} and the discussions preceding it we immediately obtain the following criterion for $Q_\Gamma(\mathfrak{R})$ to be a minimal prime ideal of $J(\mathfrak{R}).$

\begin{proposition}\label{pr44}
 Let $\Gamma\in co(\mathfrak{R})$. Then the following conditions are equivalent:
\begin{enumerate}
 \item
$Q_\Gamma(\mathfrak{R})$ is a minimal prime ideal of $J(\mathfrak{R})$;

\item
$Q_{\Gamma'}(\mathfrak{R})\not\subset Q_{\Gamma}(\mathfrak{R})$ for every $\Gamma'\in co(\mathfrak{R})$ with $\Gamma'\subset\Gamma$;

\item
there exists $r\in n\mathcal{C}(K\mathfrak{R}_0(\Gamma'))-\mathcal{C}(K\mathfrak{R}_0(\Gamma))$ such that $|\Gamma\cap \supp(r)|\leq 1$ for every $\Gamma'\in co(\mathfrak{R})$ with $\Gamma'\subset\Gamma$.
\end{enumerate}
 \end{proposition}

Next we give a simple example which demonstrates that the previous proposition could be very useful. Recall that a simple graph $G$ on vertex set $[m]$ and edge set $E$ defines a graphic arrangement $\mathcal{A}_G$ in $K^m$ as follows: let $y_1,\ldots,y_m$ be a basis for the dual space $(K^{m})^*$, then $\mathcal{A}_G=\{\ker(y_i-y_j)\mid\{i,j\}\in {E}\}$. It is obvious that $\mathcal{A}_G$ has $|{E}|$ hyperplanes and easy to show that $\rank(\mathcal{A}_G)= m-\omega(G)$, where $\omega(G)$ is the number of connected components of $G$.

\begin{example}\label{ex45}
 Let $G$ be the graph with labeled edges depicted in Figure \ref{fig1}(a). Consider the graphic arrangement $\mathcal{A}_G$. Since $G$ is connected and has 5 vertices and 8 edges, $$\dim_K F(\mathcal{A}_G)=8-(5-1)=4.$$ Let $r_1=x_1+x_5-x_8, r_2=x_2-x_5+x_6,r_3=x_3-x_6+x_7$, and $r_4=x_4-x_7+x_8 $ be the relations corresponding to the 4 triangles of $G$. Then $r_1,\ldots,r_4$ are linearly independent (since $r_i$ is the only relation involving $x_i$ for $i=1,\ldots,4$). Thus $\mathfrak{R}=\{r_1,\ldots,r_4\}$ is a basis for $F(\mathcal{A}_G)$, and so $\mathcal{A}_G$ is 2-formal. We will show that $J(\mathfrak{R})$ has only two minimal prime ideals, namely, $Q_\emptyset(\mathfrak{R})=I(\mathfrak{R})=I(\mathcal{A}_G)$ and $Q_{\{5,6,7,8\}}(\mathfrak{R})=(x_5,x_6,x_7,x_8)$. Indeed, the set $n\mathcal{C}(K\mathfrak{R})$ consists of only one relation $r_C$ which corresponds to the cycle $C=\{1,2,3,4\}$ of $G$. Let $\Gamma$ be a non empty $\mathfrak{R}$-cover. It is easily seen that $\Gamma$  contains either $\Gamma_0=\{5,6,7,8\}$ or at
least two edges of $C$. If $\Gamma\supset\Gamma_0$, then clearly $Q_\Gamma(\mathfrak{R})\supset Q_{\Gamma_0}(\mathfrak{R})=(x_5,x_6,x_7,x_8)$. If, otherwise, $\Gamma$  contains at least two edges of $C$, then $Q_\Gamma(\mathfrak{R})\supset Q_\emptyset(\mathfrak{R})$ by Proposition \ref{pr44}. Thus $Q_\Gamma(\mathfrak{R})$ is not a minimal prime ideal of $J(\mathfrak{R})$ for $\Gamma\ne \emptyset,\Gamma_0$. Finally, again by Proposition \ref{pr44}, one has $Q_{\emptyset}(\mathfrak{R})\not\subset Q_{\Gamma_0}(\mathfrak{R})$ since $\mathfrak{R}_0(\Gamma_0)=\emptyset$ and $\Gamma_0\cap\supp(r_C)=\emptyset$. Hence $Q_{\emptyset}(\mathfrak{R})$ and $Q_{\Gamma_0}(\mathfrak{R})$ are both minimal prime ideals of $J(\mathfrak{R})$.
\end{example}

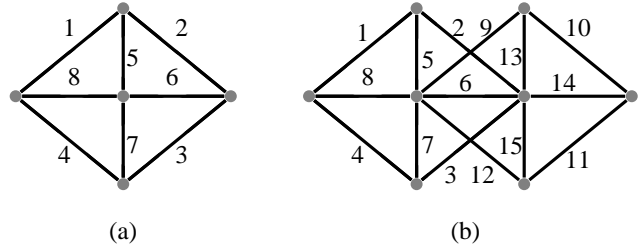
\begin{figure}[ht]

\begin{center}

\begin{tikzpicture} [scale = .13, very thick = 10mm]

  \node (n4) at (4,-3)  [Cgray] {};
  \node (n1) at (4,15) [Cgray] {};
  \node (n2) at (-7,6)  [Cgray] {};
  \node (n3) at (15,6)  [Cgray] {};
\node (n5) at (4,6)  [Cgray] {};
  \foreach \from/\to in {n4/n2,n1/n3}
    \draw[] (\from) -- (\to);
\foreach \from/\to in {n2/n1,n4/n3,n5/n2,n5/n3,n5/n1,n5/n4}
    \draw[] (\from) -- (\to);

 \node (m4) at (-1,8)  [Cwhite] {$8$};  \node (m5) at (9,8)  [Cwhite] {$6$};
   \node (m1) at (-1.5,13) [Cwhite] {$1$};
 \node (m1) at (5,10) [Cwhite] {$5$}; \node (m1) at (5,1) [Cwhite] {$7$};
  \node (m1) at (10,13) [Cwhite] {$2$};
 \node (m1) at (10,0) [Cwhite] {$3$};
\node (m1) at (-2,0) [Cwhite] {$4$};
 \foreach \from/\to in {n2/n1,n2/n5,n2/n4, n4/n3,n4/n5,n5/n1,n5/n3,n1/n3}
\draw[black][] (\from) -- (\to);

 \node (l) at (4,-8)  [Cwhite] {(a)};

  \node (n4) at (34,-3)  [Cgray] {};    \node (l) at (39,-8)  [Cwhite] {(b)};
  \node (n1) at (34,15) [Cgray] {};
  \node (n2) at (23,6)  [Cgray] {};
  \node (n3) at (45,6)  [Cgray] {};
\node (n5) at (34,6)  [Cgray] {};
  \foreach \from/\to in {n4/n2,n1/n3}
    \draw[] (\from) -- (\to);
\foreach \from/\to in {n2/n1,n4/n3,n5/n2,n5/n3,n5/n1,n5/n4}
    \draw[] (\from) -- (\to);

\node (m1) at (38.3,13) [Cwhite] {$2$}; \node (m1) at (35.2,9.8) [Cwhite] {$5$};  \node (m1) at (35.2,1) [Cwhite] {$7$};
\node (m1) at (41.1,13) [Cwhite] {$9$};  \node (m1) at (50.5,13) [Cwhite] {$10$};
\node (m1) at (50.5,-0.5) [Cwhite] {$11$};
\node (m1) at (40.7,-2) [Cwhite] {$12$};

 \node (m4) at (29,8)  [Cwhite] {$8$};
\node (m5) at (39,7.4)  [Cwhite] {$6$};
\node (m5) at (49,7.4)  [Cwhite] {$14$};
  \node (m1) at (28.5,12.5) [Cwhite] {$1$};

\node (m1) at (43.6,10) [Cwhite] {$13$}; \node (m1) at (43.6,1) [Cwhite] {$15$};


 \node (m1) at (37.5,-2) [Cwhite] {$3$};

\node (m1) at (28,0) [Cwhite] {$4$};
 \foreach \from/\to in {n2/n1,n2/n5,n2/n4, n4/n3,n4/n5,n5/n1,n5/n3,n1/n3}
\draw[black][] (\from) -- (\to);

 \node (n6) at (45,15)  [Cgray] {};
  \node (n7) at (56,6) [Cgray] {};
  \node (n8) at (45,-3)  [Cgray] {};
  \foreach \from/\to in {n6/n7,n7/n8,n5/n6,n5/n8,n3/n6,n3/n8,n3/n7}
    \draw[] (\from) -- (\to);

\end{tikzpicture}

\caption{(a) Graph $G$, (b) Graph $G_2$ consisting of $2$ copies of $G$}
\label{fig1}
\end{center}

\end{figure}

To conclude this section we give a sufficient condition for $J(\mathfrak{R})$ to be a prime ideal. We see from Equation \eqref{eq34} that $I(\mathfrak{R})$ itself is a primary component of $J(\mathfrak{R})$. Therefore, $J(\mathfrak{R})$ is a prime ideal if and only if $J(\mathfrak{R})=I(\mathfrak{R})$, or equivalently, if and only if $I(\mathfrak{R})=Q_{\emptyset}(\mathfrak{R})$ is the unique minimal prime ideal of $J(\mathfrak{R})$. By Proposition \ref{pr44}, this will be the case if $n\mathcal{C}(K\mathfrak{R})=n\mathcal{C}(K\mathfrak{R}_0(\emptyset))=\emptyset$. Thus, especially, if $\mathfrak{R}$ is finite and its elements admit an enumeration $r_1,\ldots,r_m$ such that $|\supp(r_i)\cap(\bigcup_{j=1}^{i-1}\supp(r_j))|\leq1$ for $i=2,\ldots,m$, then $J(\mathfrak{R})$ is prime. When $\mathfrak{R}$ satisfies this condition, it will be called \emph{simple}. A reformulation of simpleness, which is more convenient for an application in the next section, will be derived below.

Let $\mathfrak{R}$ be a finite set of relations such that $|\supp(r)\cap\supp(r')|\leq1$ for every distinct $r,r'\in \mathfrak{R}$. We define the \emph{intersection graph} $\mathcal{G}(\mathfrak{R})$ on vertex set $\mathfrak{R}$ as follows: $\{r,r'\}$ is an edge of $\mathcal{G}(\mathfrak{R})$ if and only if $\supp(r)\cap\supp(r')\ne\emptyset.$ Let us label the edge $\{r,r'\}$ of $\mathcal{G}(\mathfrak{R})$ with $\supp(r)\cap\supp(r')$. A cycle of $\mathcal{G}(\mathfrak{R})$ is called \emph{proper} if its edge labels are pairwise distinct. We say that $\mathcal{G}(\mathfrak{R})$ is \emph{quasi-acyclic} if it contains no proper cycles.

 Notice that the conditions of simpleness and quasi-acyclicity introduced above have their root in the study of the complete intersection property of the Orlik--Terao ideal \cite{LR}. The next lemma shows that the two conditions are equivalent.

\begin{lemma}\label{lm46}
Let $\mathfrak{R}$ be a finite subset of the space $F(\mathcal{A})$. Then $\mathfrak{R}$ is simple if and only if the intersection graph $\mathcal{G}(\mathfrak{R})$ is quasi-acyclic.
\end{lemma}

\begin{proof}
Assume $\mathfrak{R}$ is simple. Then obviously every subset of $\mathfrak{R}$ is also simple. If $D$ is a cycle of $\mathcal{G}(\mathfrak{R})$, then for every vertex $r\in D$ (here, $D$ is identified with its set of vertices), the support of $r$ intersects with the supports of the two adjacent vertices of $r$. Hence if $D$ is proper, then $|\supp(r)\cap(\bigcup_{r'\in D-r}\supp(r'))|\geq 2$ for every $r\in D$. This implies that $D$ is not simple, a contradiction. Thus $\mathcal{G}(\mathfrak{R})$ must be quasi-acyclic.

Conversely, assuming $\mathcal{G}(\mathfrak{R})$ is quasi-acyclic, we prove that $\mathfrak{R}$ is simple. Evidently, it suffices to show that there exists $r\in \mathfrak{R}$ such that
$$d(r):=|\supp(r)\cap(\bigcup_{r'\in \mathfrak{R}-r}\supp(r'))|\leq 1.$$
Suppose on the contrary that $d(r)\geq 2$ for all $r\in \mathfrak{R}$. Then $\mathcal{G}(\mathfrak{R})$ clearly contains a cycle $D$ with at least two distinct edges. Since $|\supp(r)\cap\supp(r')|\leq 1$ for every distinct $r,r'\in \mathfrak{R}$, the edges of $D$ with the same label must form a path. Replacing each such path of $D$ by the edge connecting the two end vertices of the path (note that this edge does exist and has the same label with the edges in the path), we get a proper cycle. This contradiction completes the proof.
\end{proof}

\begin{proposition}\label{pr47}
Let $\mathfrak{R}$ be a finite subset of the space $F(\mathcal{A})$ such that the intersection graph $\mathcal{G}(\mathfrak{R})$ is quasi-acyclic. Then $J(\mathfrak{R})=I(\mathfrak{R})$ is a prime ideal. Furthermore, $J(\mathfrak{R})$ is a complete intersection.
\end{proposition}

\begin{proof}
The first assertion is clear from Lemma \ref{lm46} and the discussion before it. The second one follows from
$\codim J(\mathfrak{R})=\codim I(\mathfrak{R})=\dim_K K\mathfrak{R}$ (see Proposition \ref{pr26}).
\end{proof}

\section{Examples}

From \cite[Theorem 2.3]{ST} (or Corollary \ref{co37}) it follows that an arrangement $\mathcal{A}$ is 2-formal if $\codim I(\mathcal{A})=\codim I_{\langle2\rangle}(\mathcal{A}).$ However, the converse is not true in general (the revised version of \cite{ST} will appear on arXiv.org). In this section we give several examples of 2-formal arrangements, including a family of graphic arrangements, for which the codimension of the quadratic Orlik--Terao ideal behaves badly. We show that there does not exist a linear bound for $\codim I(\mathcal{A})$ in terms of $\codim I_{\langle2\rangle}(\mathcal{A})$, even when $\mathcal{A}$ is 2-formal.

For simplicity, it will be assumed throughout this section that the field $K$ has characteristic 0. We begin with a modification of Yuzvinsky's example.

\begin{example}\label{ex51}
In order to show that the property of being 2-formal is not combinatorial, Yuzvinsky \cite{Y3} considered two arrangements in $K^3$, one is 2-formal and the other not, with isomorphic intersection lattices. We now define the following arrangement which shares 7 common hyperplanes with the two arrangements in \cite{Y3}:
$$\mathcal{A}=\mathbf{V}(yzw(y+z+w)(2y+z+w)(2y+3z+w)(2y+3z+4w)(y+w)(2y+2z+3w)).$$
A computation using Macaulay2 \cite{GS} shows that $I(\mathcal{A})={I_{\langle 2\rangle}(\mathcal{A})}:x_{[n]}$. Thus $\mathcal{A}$ is 2-formal by Corollary \ref{co37}. However, $\codim I_{\langle 2\rangle}(\mathcal{A})= 5<6=\codim I(\mathcal{A})$. In this example, $I_{\langle 2\rangle}(\mathcal{A})$ is not a radical ideal.
\end{example}

The next example illustrates that the codimension of the quadratic Orlik--Terao ideal may behave badly even for 2-formal graphic arrangements.

\begin{example}\label{ex52}
Figure~\ref{fig1}(b) depicts the graph $G_2$ obtained by ``gluing'' two copies of the graph $G$ in Figure~\ref{fig1}(a) along one ``inner'' edge of each copy. (Formally, $G_2$ is a parallel connection of the two copies of $G$; see \cite[Section 7.1]{O}.) Consider the graphic arrangement $\mathcal{A}_{G_2}$ of $G_2$. Let $\mathfrak{R}$ be the set of 8 relations corresponding to the 8 triangles of $G_2$. Then arguing similarly as in Example \ref{ex45}, $\mathfrak{R}$ is a basis for $F(\mathcal{A}_{G_2})$. Thus $\mathcal{A}_{G_2}$ is 2-formal. Evidently, $\Gamma=\{5,6,7,8,13,14,15\}\in co(\mathfrak{R})$ and $Q_\Gamma(\mathfrak{R})=(x_i\mid i\in\Gamma)$. It follows that
\[
\codim I_{\langle 2\rangle}(\mathcal{A}_{G_2})=\codim J(\mathfrak{R})\leq \codim Q_\Gamma(\mathfrak{R})=7.
\]
On the other hand, $\codim I(\mathcal{A}_{G_2})= \dim _K F(\mathcal{A}_{G_2})=8$. Hence
$$\codim I(\mathcal{A}_{G_2})\geq \codim I_{\langle 2\rangle}(\mathcal{A}_{G_2})+1.$$
More generally, for $k\geq 2$ one can consider the graph $G_k$ obtained by ``gluing'' $k$ copies of $G$, as depicted in Figure~\ref{fig2}. Then a straightforward extension of the above argument shows that $\mathcal{A}_{G_k}$ is 2-formal, $\codim I(\mathcal{A}_{G_k})= 4k$, and $\codim I_{\langle 2\rangle}(\mathcal{A}_{G_k})\leq 3k+1$. Thus, in this case,
$$\codim I(\mathcal{A}_{G_k})\geq \codim I_{\langle 2\rangle}(\mathcal{A}_{G_k})+k-1.$$
\end{example}

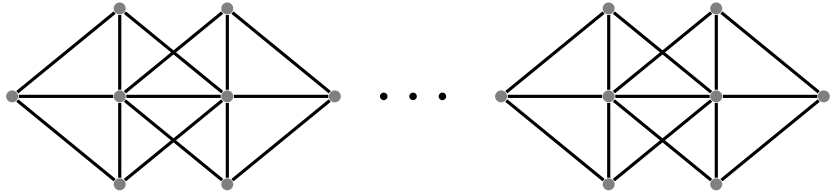
\begin{figure}[ht]

\begin{center}

\begin{tikzpicture} [scale = .13, very thick = 10mm]



 \node (n4) at (74,-3)  [Cgray] {};
  \node (n1) at (74,15) [Cgray] {};
  \node (n2) at (63,6)  [Cgray] {};
  \node (n3) at (85,6)  [Cgray] {};
\node (n5) at (74,6)  [Cgray] {};


  \foreach \from/\to in {n4/n2,n1/n3}
    \draw[] (\from) -- (\to);
\foreach \from/\to in {n2/n1,n4/n3,n5/n2,n5/n3,n5/n1,n5/n4}
    \draw[] (\from) -- (\to);

 \node (n6) at (85,15)  [Cgray] {};
  \node (n7) at (96,6) [Cgray] {};
  \node (n8) at (85,-3)  [Cgray] {};
  \foreach \from/\to in {n6/n7,n7/n8,n5/n6,n5/n8,n3/n6,n3/n8,n3/n7}
    \draw[] (\from) -- (\to);


  \node (n2) at (101,6)  [Cblack2] {};
  \node (n2) at (104,6)  [Cblack2] {};
  \node (n2) at (107,6)  [Cblack2] {};


 \node (n4) at (124,-3)  [Cgray] {};
  \node (n1) at (124,15) [Cgray] {};
  \node (n2) at (113,6)  [Cgray] {};
  \node (n3) at (135,6)  [Cgray] {};
\node (n5) at (124,6)  [Cgray] {};

  \foreach \from/\to in {n4/n2,n1/n3}
    \draw[] (\from) -- (\to);
\foreach \from/\to in {n2/n1,n4/n3,n5/n2,n5/n3,n5/n1,n5/n4}
    \draw[] (\from) -- (\to);

 \node (n6) at (135,15)  [Cgray] {};
  \node (n7) at (146,6) [Cgray] {};
  \node (n8) at (135,-3)  [Cgray] {};
  \foreach \from/\to in {n6/n7,n7/n8,n5/n6,n5/n8,n3/n6,n3/n8,n3/n7}
    \draw[] (\from) -- (\to);


\end{tikzpicture}
\caption{Graph $G_k$ consisting of $k$ copies of $G$}
\label{fig2}
\end{center}

\end{figure}

Our last example is aimed to show that even in the class of 2-formal arrangements one cannot expect a linear bound for the codimension of the Orlik--Terao ideal in terms of the codimension of the quadratic Orlik--Terao ideal.

\begin{example}\label{ex53}

Let $G$ be a graph on vertex set $[m]$ and edge set $E$ with $|E|=l$. Let $y_1,\ldots,y_m$ be a basis for the dual space $(K^m)^*$ of $K^m$. Consider the arrangement $\mathcal{B}_G$ in $K^m$ consisting of $m$ hyperplanes indexed by the vertices and $l$ hyperplanes indexed by the edges of $G$ as follows:
$$H_i= \ker y_i\ \ \text{for}\ \ i\in [m]\ \ \text{and}\ \ H_e=\ker(y_i+y_j)\ \text{for}\ e=\{i,j\}\in E.$$
Since $\mathcal{B}_G$ has rank $m$, $\dim_KF(\mathcal{B}_G)=(m+l)-m=l$. For each edge $e=\{i,j\}\in E$, the 3-circuit $\{H_i,H_j,H_e\}$ gives rise to a relation $r_e=x_e-x_i-x_j$. Let $\mathfrak{R}_G=\{r_e\mid e\in E\}$. Then $\mathfrak{R}_G$ is linearly independent because for every $e\in E$, $r_e$ is the only relation in $\mathfrak{R}_G$ involving $x_e$. Thus $\mathfrak{R}_G$ is a basis for $F(\mathcal{B}_G)$, and hence $\mathcal{B}_G$ is a 2-formal arrangement.

Since $\ch(K)=0$, it is apparent that every 3-circuit of $\mathcal{B}_G$ is of the form $\{H_i,H_j,H_e\}$ for some $e=\{i,j\}\in E$. Thus $I_{\langle 2\rangle}(\mathcal{B}_{G})=J(\mathfrak{R}_G).$ Clearly, $\Gamma=\{1,\ldots,m\}$ is an $\mathfrak{R}_G$-cover and $Q_\Gamma(\mathfrak{R}_G)=(x_1,\ldots,x_m)$. It follows that $\codim I_{\langle 2\rangle}(\mathcal{B}_{G})\leq m$. On the other hand, $\codim I_{\langle 2\rangle}(\mathcal{B}_{G})\leq \codim I(\mathcal{B}_{G})=l$. Therefore, $\codim I_{\langle 2\rangle}(\mathcal{B}_{G})\leq\min\{m,l\}$.

\begin{claim}
Suppose $G$ is a connected graph. Then $\codim I_{\langle 2\rangle}(\mathcal{B}_{G})=\min\{m,l\}$.
\end{claim}

\begin{proof}
 Let $c=l-m+1$ be the cyclomatic number (or nullity) of $G$. Note that $c\geq0$. We distinguish the following cases:

 \emph{Case 1}: $c=0$, i.e., $l=m-1$ and $G$ is a tree. In this case, the intersection graph $\mathcal{G}(\mathfrak{R}_G)$ of $\mathfrak{R}_G$ is quasi-acyclic. Indeed, we first have $|\supp(r_e)\cap\supp(r_{e'})|\leq1$ for every distinct edges $e,e'$ of $E$. Suppose that $D$ is a proper cycle of $\mathcal{G}(\mathfrak{R}_G)$ with vertices $r_{e_1},\ldots,r_{e_k}$ and edge labels $\{i_1\},\ldots,\{i_k\}$, where $\{i_j\}=\supp(r_{e_j})\cap\supp(r_{e_{j+1}})$ for $j=1,\ldots,k$ (here, by convention, $e_{k+1}=e_1$). Then clearly $e_{j+1}=\{i_j,i_{j+1}\}$ for $j=1,\ldots,k$. It follows that the edges $e_1,\ldots,e_k$ form a cycle of $G$. However, this is impossible since $G$ is a tree. Thus $\mathcal{G}(\mathfrak{R}_G)$ must be quasi-acyclic. Now by Proposition \ref{pr47}, $I_{\langle 2\rangle}(\mathcal{B}_{G})=I(\mathcal{B}_{G})$ and so
 $$\codim I_{\langle 2\rangle}(\mathcal{B}_{G})=\codim I(\mathcal{B}_{G})=l=\min\{m,l\}.$$

 \emph{Case 2}: $c=1$, i.e., $l=m$. Let $G'$ be a spanning tree of $G$. Then $G'$ is obtained from $G$ by deleting some edge $e\in E$. We have $\mathfrak{R}_G=\mathfrak{R}_{G'}\cup\{r_e\}$ and $I_{\langle 2\rangle}(\mathcal{B}_{G})=I_{\langle 2\rangle}(\mathcal{B}_{G'})+(\iota(r_e))$. Since $\mathfrak{R}_G$ is linearly independent, $r_e\not\in K\mathfrak{R}_{G'}$. So by Lemma \ref{lm32}, $\iota(r_e)\not\in I_{\langle 2\rangle}(\mathcal{B}_{G'})$. Now according to Case 1, $I_{\langle 2\rangle}(\mathcal{B}_{G'})$ is a prime ideal and $\codim I_{\langle 2\rangle}(\mathcal{B}_{G'})=l-1$.  It follows that
 $$\codim I_{\langle 2\rangle}(\mathcal{B}_{G})=\codim I_{\langle 2\rangle}(\mathcal{B}_{G'})+1=l=m.$$

\emph{Case 3}: $c>1$, i.e., $l>m$. Let $G''$ be the subgraph obtained from $G$ by deleting $l-m$ edges. Then the cyclomatic number of $G''$ is 1, and it follows from Case 2 that
 $$\codim I_{\langle 2\rangle}(\mathcal{B}_{G})\geq\codim I_{\langle 2\rangle}(\mathcal{B}_{G''})=m=\min\{m,l\}.$$
 Since the reverse inequality is already known, the claim has been proved.
\end{proof}
\end{example}

As an easy consequence of the previous example, we obtain:

\begin{corollary}\label{co54}
 For every integer $k\geq1$, there exists a 2-formal arrangement $\mathcal{A}$ such that
 $$\codim I(\mathcal{A})=k\ {\cdot}\ \codim I_{\langle 2\rangle}(\mathcal{A}).$$
\end{corollary}

\begin{proof}
Let $G$ be the complete graph on $2k+1$ vertices. According to Example \ref{ex53}, one may take $\mathcal{A}=\mathcal{B}_{G}.$
\end{proof}


As we have seen, for an arrangement $\mathcal{A}$ the property of being 2-formal is not good enough to afford the equality $\codim I(\mathcal{A})=\codim I_{\langle 2\rangle}(\mathcal{A}).$ A natural question then arises: what does this equality mean? Or more specifically, can this equality be followed from other properties which are stronger than 2-formality? For an arrangement $\mathcal{A}$ with quadratic Orlik--Solomon algebra, the latter question is easy to answer affirmatively: according to \cite[Corollary 2.17]{F}, the underlying matroid of $\mathcal{A}$ is line-closed, and so by \cite[Proposition 21]{SSV}, $ \mathbf{V}(I(\mathcal{A}))=\mathbf{V}(I_{\langle 2\rangle}(\mathcal{A}))$. In particular, the answer for rational $K(\pi, 1)$ arrangements is also affirmative because these arrangements have quadratic Orlik--Solomon algebras; see \cite{FR1}. It would be interesting to know the answer for free and $K(\pi, 1)$ arrangements.

\begin{question}\label{qu55}
 Let $\mathcal{A}$ be a free or $K(\pi, 1)$ arrangement. Is it true that $\codim I(\mathcal{A})=\codim I_{\langle 2\rangle}(\mathcal{A})$?
\end{question}

\noindent\textbf{Acknowledgement.} The authors wish to thank Tim R\"{o}mer for helpful discussions and valuable suggestions. They are grateful to Hal Schenck for valuable comments on an earlier draft of this paper and for sharing unpublished versions of \cite{ST}.

\end{document}